\documentclass[11 pt]{amsart}
\usepackage{amsmath,amsthm,amsfonts,amssymb,amstext, latexsym, amscd}

\newtheorem{theorem}{Theorem}[section]
\newtheorem{lem}[theorem]{Lemma}
\newtheorem{prop}[theorem]{Proposition}
\newtheorem{cor}[theorem]{Corollary}

\theoremstyle{definition}
\newtheorem{definition}[theorem]{Definition}

\theoremstyle{remark}
\newtheorem{remark}[theorem]{Remark}
\numberwithin{equation}{section}

\newcommand{\ra}{\rightarrow}

\newcommand{\C}{{\mathbb{C}}}
\newcommand{\HH}{{\mathbb{H}}}
\newcommand{\CP}{{\mathbb{CP}}}

\newcommand{\HP}{{\mathbb{HP}}}
\newcommand{\RP}{{\mathbb{RP}}}

\newcommand{\KK}{\mathbb{K}}
\newcommand{\KP}{\mathbb{KP}}
\newcommand{\Z}{\mathbb{Z}}

\newcommand{\dimm}{\text{dim}}

\newcommand{\lb}{\langle}
\newcommand{\rb}{\rangle}

\newcommand{\mg}{\mathfrak{g}}
\newcommand{\mh}{\mathfrak{h}}
\newcommand{\mk}{\mathfrak{k}}
\newcommand{\mm}{\mathfrak{m}}

\newcommand{\ma}{\mathfrak{a}}
\newcommand{\ml}{\mathfrak{l}}
\newcommand{\mf}{\mathfrak{f}}
\renewcommand{\mp}{\mathfrak{p}}

\newcommand{\mt}{\mathfrak{t}}

\newcommand{\ad}{\text{ad}}

\newcommand{\R}{\mathbb{R}}

\begin{document}

\newcommand{\spacing}[1]{\renewcommand{\baselinestretch}{#1}\large\normalsize}
\spacing{1.14}

\title{A note on quasi-positive curvature conditions}

\thanks{$^\ast$Supported in part by NSF grant DMS--0902942.}
\author {Megan M. Kerr}
\author {Kristopher Tapp$^\ast$}
\begin{abstract} We classify the triples $H \subset K \subset G$ of nested compact Lie groups which satisfy the ``positive triple'' condition that was shown in~\cite{Tapp} to ensure that $G/H$ admits a metric with quasi-positive curvature.  A few new examples of spaces that admit quasi-positively curved metrics emerge from this classification; namely,  a $\CP^2$-bundle over $S^6$, a $B^7$-bundle over $\HP^2$, a $\CP^{2n-1}$-bundle over $\HP^{n}$ for each $n\geq 2$, and a family of finite quotients of $T^1S^6$.
\end{abstract}
\address{Department of Mathematics\\ Wellesley College \\ Wellesley, MA 02481}
\email{mkerr@wellesley.edu}
\address{Department of Mathematics\\ Saint Joseph's University\\
5600 City Ave.\\
Philadelphia, PA 19131}
\email{ktapp@sju.edu}

\maketitle

\section{Introduction}
The study of simply connected Riemannian manifolds with positive sectional curvature has long been considered important in Riemannian geometry.  Nevertheless, very few examples have been found.  Aside from the rank one symmetric spaces, the first examples discovered were homogeneous spaces of dimension 6, 7, 12, 13 and 24 due to Berger~\cite{Berger}, Wallach~\cite{Wallach}, and Aloff-Wallach~\cite{AW}, followed by biquotients of dimension 6, 7 and 13 due to Eschenburg~\cite{E1} and Bazaikin~\cite{Ba}.  More recently, a 7-dimensional cohomogeneity-one manifold was endowed with positive curvature by Grove-Verdiani-Ziller~\cite{GVZ} and independently by Dearicott~\cite{D}, and an exotic sphere with positive curvature was constructed by Petersen-Wilhelm~\cite{PW}.  See Ziller~\cite{Ziller} for a survey of known examples and their constructions.

In contrast, there are many more manifolds which are known to admit Riemannian metrics of nonnegative sectional curvature, including every compact homogeneous space.  In order to better understand the gap between nonnegative and positive curvature, recent attention has focused on the construction of nonnegatively curved manifolds with strictly positive curvature either at a point (called \emph{quasi-positive curvature}) or on an open dense set of points (called \emph{almost-positive curvature}).  For example, Gromoll and Meyer discovered an exotic sphere with quasi-positive curvature~\cite{GM}, which was later proven to admit almost-positive curvature~\cite{Fred},\cite{Esch},\cite{Esch2}.  Petersen and Wilhelm endowed $T^1S^4$ and a 6-dimensional quotient of $T^1S^4$ with almost-positive curvature~\cite{PW}.  Kerin discovered many examples of almost-positive curvature among Eschenburg and Bazaikin biquotients and other spaces~\cite{K1},\cite{K2}.

Wilking discovered several infinite families of homogeneous bundles whose total spaces admit almost-positive curvature~\cite{Wilking}.  His examples include the projective tangent bundles of all projective spaces.  The second author then found some related infinite families of homogeneous bundles whose total spaces admit quasi-positive curvature, including the unit tangent bundles of all projective spaces~\cite{Tapp}.  These homogeneous bundle examples from~\cite{Wilking} and~\cite{Tapp}, which we will enumerate in Section 2, have in common  that they come from \emph{positive triples}, defined as follows:
\begin{definition}\label{D:postrip} A triple $H\subset K\subset G$ of nested compact Lie groups (with Lie algebras
$\mh\subsetneq \mk\subsetneq \mg$) is called a \emph{positive triple} if the following three conditions are satisfied:
\begin{enumerate}
\item[(i)] No pair of linearly independent vectors in $\mp=\mg\ominus\mk$ commute.
\item[(ii)]  No pair of linearly independent vectors in $\mm=\mk\ominus\mh$ commute.
\item[(iii)]  There exists $A\in\mg$ such that $[Z^\mm,[A,W]^\mk]\neq 0$ for all linearly independent vectors $Z\in\mm\oplus\mp$ and $W\in\mp$ for which $[Z,W]=0$.
\end{enumerate}
In this case, we also refer to the triple of Lie algebras $\mh\subset\mk\subset\mg$ as a \emph{positive triple}.  Here  ``$\ominus$'' denotes the orthogonal compliment, and superscripts denote orthogonal projections, all with respect to a bi-invariant metric on $G$.
\end{definition}
Notice that the homogeneous space $M=G/H$ is topologically a bundle over $B=G/K$ with fiber $F=K/H$.
The first two conditions say that the base space $B$ and the fiber $F$ both have positive sectional curvature with respect to their normal homogeneous metrics.  The third condition represents a derivative guarantee that positively curved points in $M$ (with respect to a certain natural nonnegatively curved metric) are reached by moving away from the identity coset in the direction of $A$.  Therefore:
\begin{theorem}\cite{Tapp}\label{ipad} If $H\subset K\subset G$ is a positive triple, then $M=G/H$ admits an inhomogeneous metric with quasi-positive curvature.
\end{theorem}
The goal of this paper is to classify all positive triples.  The following new examples arise from the classification:
\begin{theorem}\label{T:E}
The following are positive triples:
\begin{enumerate}
\item $SU(3)\subset G_2\subset Spin(7)$,
\item $SU(2)\subset SU(3)\subset G_2$,
\item $U(2)\subset SU(3)\subset G_2$,
\item $\{\text{diag}(z,\eta,A)\mid z\in Sp(1),\eta\in U(1), A\in Sp(n-1)\}\subset Sp(1)\times Sp(n)\subset Sp(n+1)$,
$(n\geq 2)$,
\item $SO(3)_{max}\times Sp(1)\subset Sp(2)\times Sp(1)\subset Sp(3)$.
\end{enumerate}
\end{theorem}
Triple (1) gives $M=T^1 S^7$, and Triple (2) gives $M=T^1 S^6$, both of which were already known to admit almost-positive curvature, so these examples are not topologically new.  Triple (3) is obtained from (2) by enlarging $H$, which preserves the positivity property; it gives a $\CP^2$-bundle over $S^6$ which is a new example of quasi-positive curvature.  Triple (4) gives a $\CP^{2n-1}$-bundle over $\HP^{n}$ for each $n \geq 2$; it is a family of new examples.  Triple (5) gives a bundle over $\HP^2$ whose fiber is the Berger space $B^7$.  It is also a new example.

Even though Triple (2) yields the known example $M=T^1S^6$, the resultant new metric on $T^1S^6$ admits an isometric free $S^1$-action, whose quotient is the above-mentioned $\CP^2$-bundle over $S^6$.  Consider the sub-action by the finite cyclic group $\Z_k\subset S^1$:
\begin{cor}\label{zk}
For each positive integer $k$, the quotient of $T^1S^6$ by a free $\Z_k$-action admits quasi-positive curvature.
\end{cor}

This paper is organized as follows.  In Section 2, we summarize the previously known examples of positive triples.  In Section 3, we prove Theorem~\ref{T:E} and Corollary~\ref{zk}.  In Sections 4 and 5, we classify all positive triples, proving essentially that the examples of Theorem~\ref{T:E} are the only positive triples beyond the previously known examples.

\section{Previously known examples of positive triples}
In this section, we summarize the previously known examples of positive triples from~\cite{Tapp}, and the resultant examples of quasi-positively curved manifolds, several of which were first proven in~\cite{Wilking} to have the stronger property of almost-positive curvature.

First, observe that whenever $G/H$ has positive curvature with respect to a normal homogeneous metric, and $K$ is any intermediate subgroup between $H$ and $G$ (with $\text{dim}(H)<\text{dim}(K)<\text{dim}(G)$) the triple $H\subset K\subset G$ is easily seen to be a positive triple.

To describe further examples, we let $\KK\in\{\R,\C,\HH\}$ and let $G(n)$ denote $O(n)$, $U(n)$ or $Sp(n)$ respectively.  The unit tangent bundle of the corresponding projective space, $T^1\KP^n$, comes from the triple
\begin{equation}\label{T1}
\{\text{diag}(z,z,A)\mid z\in G(1), A\in G(n-1)\}\subset G(1)\times G(n)\subset G(n+1).
\end{equation}
This is a positive triple, and therefore $T^1\KP^n$ admits quasi-positive curvature.  Similarly, the unit tangent bundle of the Cayley Plane, $T^1\mathbb{OP}^2$, admits quasi-positive curvature because the following is a positive triple:
\begin{equation}\label{T2}
Spin^+(7)\subset Spin(9)\subset F_4.
\end{equation}
The projective tangent bundle of $\KP^n$, denoted $P_{\KK}T\KP^n$, comes from the triple
\begin{equation}\label{T3}
\{\text{diag}(z_1,z_2,A)\mid z_i\in G(1), A\in G(n-1)\}\subset G(1)\times G(n)\subset G(n+1).
\end{equation}
Since this triple is obtained from Triple~\ref{T1} by enlarging $H$, which
preserves the positivity property, $P_{\KK}T\KP^n$ admits quasi-positive curvature.  Moreover, $P_{\KK}T\KP^n$ is known to have almost-positive curvature.
The projective tangent bundle of the Cayley Plane, $P_{\mathbb{O}}T\mathbb{OP}^2$, corresponds to
\begin{equation}\label{T2P}
Spin(8)\subset Spin(9)\subset F_4.
\end{equation}
This triple is positive because it is obtained from Triple~\ref{T2} by enlarging $H$. In fact, $P_{\mathbb{O}}T\mathbb{OP}^2$ is one of three flag manifolds proven by Wallach in~\cite{Wallach} to admit a homogeneous metric of positive curvature.  It is sometimes denoted $W^{24}$.  Wallach's other two flag manifolds, $W^6=P_\C T\CP^2$ and $W^{12}=P_\HH T\HP^2$, come from the $n=2$ cases of  Triple~\ref{T3} with $\KK\in\{\C,\HH\}$.

In the triple for $T^1\HP^n$ above (\ref{T1} with $\KK=\HH$), there is a different  embedding of $H$ in $K$ which also yields a positive triple:
\begin{equation}\label{T4}
\{\text{diag}(z,1,A)\mid z\in Sp(1), A\in Sp(n-1)\}\subset Sp(1)\times Sp(n)\subset Sp(n+1) \,\,\,\,\,\, (n\geq 2).
\end{equation}
We see that the resultant $S^{4n-1}$-bundle over $\HP^n$ has quasi-positive curvature.

In the triple for $T^1\CP^n$ above  (\ref{T1} with $\KK=\C$), there is an infinite family of alternative embeddings of $H$ in $K$ parameterized by $k,l\in\Z$:
\begin{equation}\label{T5}
\{\text{diag}(z^k,z^l,A)\mid z\in U(1), A\in U(n-1)\}\subset U(1)\times U(n)\subset U(n+1),\,\,\,\,\,\,(n\geq 2).
\end{equation}
This is a positive triple if $k$ and $l$ are not both zero.  The resultant lens space bundle over $\CP^n$ therefore has quasi-positive curvature; moreover, it has almost-positive cuvature when $k\cdot l<0$.  Notice that when $n=2$, the family of total spaces corresponding to these triples includes the Aloff-Wallach spaces that were proven in~\cite{AW} to admit homogeneous metrics with positive curvature.

In Triple~\ref{T5}, notice that $SU(n+1)$ acts transitively on $G/H$ because every coset intersects SU(n+1).  If we intersect the three groups with SU(n+1) we obtain
\begin{equation}\label{T5p}
\{\text{diag}(z^k,z^l,A)\mid z\in U(1), A\in U(n-1)\}\cap SU(n+1)\subset S(U(1)\times U(n))\subset SU(n+1),
\end{equation}
which is also a positive triple when $k$ and $l$ are not both zero.  Triples~\ref{T5} and~\ref{T5p} have diffeomorphic total spaces.

Similarly, the three groups in the triple for $P_{\C}T\CP^n$ (\ref{T3} with $\KK=\C$) can be intersected with $SU(n+1)$ to obtain the following alternative positive triple for $P_{\C}T\CP^n$:
\begin{equation}\label{green}
S(U(1)\times U(1)\times U(n-1))\subset S(U(1)\times U(n))\subset SU(n+1).
\end{equation}

This completes our summary of the previously known examples of positive triples.
\section{New examples of positive triples}
In this section, we prove Theorem~\ref{T:E}.  Triple (4) from this theorem is a positive triple simply because it is obtained from Triple $\ref{T4}$ by enlarging $H$, which obviously preserves the positivity property.  Similarly, Triple (3) is obtained from Triple (2) by enlarging $H$. It remains only to verify that Triples (1), (2) and (5) are positive triples.

\begin{definition} Let $\mh\subset\mk\subset\mg$ be a triple of compact Lie algebras.  A vector in $\mp=\mg\ominus\mk$ is called \emph{fat} if it does not commute with any nonzero vectors in $\mm=\mk\ominus\mh$.  It is called \emph{strongly fat} if it does not commute with any vectors in $\mm\oplus\mp$ except for scalar multiples of itself.
\end{definition}

\begin{prop}\label{help}
Suppose $H\subset K\subset G$ are nested compact Lie groups with Lie algebras $\mh\subsetneq\mk\subsetneq\mg$.  Assume that the following conditions are satisfied:
\begin{enumerate}
\item[(i)]  No pair of linearly independent vectors in $\mp=\mg\ominus\mk$ commute.
\item[(ii)]  No pair of linearly independent vectors in $\mm=\mk\ominus\mh$ commute.
\item[(iii)]  The isotropy action of $K$ on $\mp$ acts transitively on the unit sphere in $\mp$.
\item[(iv)]  There exists a vector $A\in\mp$ which is fat.
\end{enumerate}
Then $H\subset K\subset G$ is a positive triple.
\end{prop}
This proposition was proven in~\cite{Tapp} under the added hypothesis that $(G,K)$ is a symmetric pair.
\begin{proof}
Let $A\in\mp$ be a fat vector.  For all nonzero $X\in\mm$, we know $[X,A]\neq 0$, so:
$$0\neq g_0([X,A],[X,A]) = g_0([X,[X,A]],A).$$
In particular, $[X,[X,A]] \neq 0$  for all nonzero $X\in\mm$.

Now let $Z\in\mm\oplus\mp$ and $W\in\mp$ be linearly independent vectors which commute, so that
$$0 = [Z,W] = [Z^\mm,W] + [Z^\mp,W].$$
Since $[Z^\mm,W]\in\mp$, we see that $0=[Z,W]^\mk = [Z^\mp,W]^\mk$.
In particular, for all $Y\in\mk$, we have
$$0 = \lb Y,[Z^\mp,W]\rb = \lb [Y,Z^\mp],W\rb.$$
By hypothesis, the isotropy representation is transitive; this implies that $Z^\mp$ is parallel to $W$.  Thus,
$X:=Z^\mm\neq 0$ and $[X,W] = 0$.

It will suffice to show that the following is nonzero:
$$[Z^\mm,[A,W]^\mk] = [X,[A,W]-[A,W]^\mp] =-[W,[X,A]] - [X,[A,W]^\mp] \text{ (Jacobi identity}).$$
To show that this is nonzero, it suffices to show that its $\mk$-component is nonzero, equivalently, to show that
$[W,[X,A]]^\mk$ is nonzero.  But, since the isotropy representation is transitive, if $[W,[X,A]]^\mk=0$, then $W$ would be parallel to $[X,A]$ (by the above argument), which would mean that $[X,[X,A]]=[X,W]=0$, which is a contradiction.
\end{proof}
It remains to establish the positivity of Triples (1), (2), and (5) from Theorem~\ref{T:Main}.  By Proposition~\ref{help}, it will suffice in each case to prove that a fat vector exists.
\begin{prop} $SO(3)_{max}\times Sp(1)\subset Sp(2)\times Sp(1)\subset Sp(3)$ is a positive triple.
\end{prop}

\begin{proof} The corresponding Lie algebra triple is $so(3)_{max}\oplus sp(1)\subset sp(2)\oplus sp(1)\subset sp(3)$.
The following basis for the maximal $so(3)$ in $sp(2)$ is found in~\cite{WP}:
$$so(3)_{max} = \text{span}_\R\left\{
  \left(\begin{matrix} \frac 32 i & 0  \\ 0 & \frac 12 i\end{matrix}\right),
  \left(\begin{matrix} 0  & \frac 12 \sqrt{3}  \\ -\frac 12 \sqrt{3} & j\end{matrix}\right),
  \left(\begin{matrix} 0  & \frac 12 \sqrt{3} i  \\ \frac 12 \sqrt{3} i & k\end{matrix}\right)
       \right\}.$$
We claim that the vector $A=(1,0)\in\HH^2\cong\mp$ is fat.  Let $\mk'$ denote the $sp(2)$-factor of $\mk=sp(2)\oplus sp(1)$.  Let $\mf\subset\mk'$ denote the kernel of $\ad_A:\mk'\ra\mp$.  Since the isotropy action of the $Sp(2)$-factor of $K$ on $\HH^2\cong\mp$ is the standard action, we have
$$\mf = \{\text{diag}(0,q)\mid q\in sp(1)\}.$$
No nonzero vector in $\mf$ is orthogonal to all three vectors of the above basis for $so(3)_{\max}$.  Therefore, $\mf\cap\mm=\{0\}$, so $A$ is fat.
\end{proof}

\begin{prop} Both $SU(2)\subset SU(3)\subset G_2$ and $SU(3)\subset G_2\subset Spin(7)$ are positive triples.
\end{prop}
\begin{proof} In each case, we have a triple $H\subset K\subset G$ for which the base space $B=G/K$ is a sphere $S^{\ell}$ with $\ell\in\{6,7\}$. The natural transitive action of $G$ on $S^{\ell}$ induces a transitive action of $G$ on $T^1S^{\ell}$.  Choose $p\in S^{\ell}$ to be a point whose isotropy group is $K$.   We know we can find a vector
$A\in T_p S^{\ell}\cong\mp$ whose isotropy group is exactly $H$.  In other words, the kernel of
$\text{ad}_A:\mk\ra\mp$ equals $\mh$. This guarantees that the vector $A$ in $\mp$ is a fat vector.
\end{proof}

\begin{proof}[Proof of Corollary~\ref{zk}]
If $H\subset K\subset G$ is a positive triple, it was proven in~\cite{Tapp} that $M=G/H$ admits a metric of quasi-positive curvature.  By ~\cite[Remark~2.1]{Tapp}, the group $K$ acts isometrically on $M$ with respect to this metric.  A subgroup $L\subset K$ acts freely if and only if $L\cap(g^{-1}\cdot H\cdot g)=\{I\}$ for all $g\in G$.

For the positive triple $SU(2)\subset SU(3)\subset G_2$, let $L$ be the center of $U(2)$ embedded in $SU(3)$; that is, $L=\{\text{diag}(e^{-2it},e^{it},e^{it})\mid t\in[0,2\pi)\}\subset SU(3)$.  Notice that $L$ acts \emph{almost} freely on $M=G_2/SU(2)$ with ineffective kernel equal to the cyclic group $\Z_2=\{\text{diag}(e^{-2it},e^{it},e^{it})\mid t\in\{0,\pi\}\}\subset L$. This is because $L$ commutes with $H$, and $L\cap H=\Z_2$. Thus, our quasi-positively curved metric on $T^1S^6$ is invariant under the free action of $L/\Z_2\cong S^1$, and therefore also under the subaction of the finite cyclic group $\Z_k\subset S^1$ for each positive integer $k$.
\end{proof}
\section{Preliminary Results}

In the section, we gather some definitions and results that will be used in the next section for stating and proving our main theorem, which classifies all positive triples.

First, we make precise the notion of equivalence we use in our classification, namely, triple-isomorphism, defined here.

\begin{definition} A nested pair of Lie algebras $\mk\subset\mg$ is called \emph{pair-isomorphic} to another nested pair $\mk'\subset\mg'$ if there is a Lie algebra isomorphism $\phi:\mg\ra\mg'$ such that $\phi(\mk)=\mk'$.  Similarly, a nested triple $\mh\subset\mk\subset\mg$ is called \emph{triple-isomorphic} to another nested triple $\mh'\subset \mk'\subset\mg'$ if there is a Lie algebra isomorphism $\phi:\mg\ra\mg'$ such that $\phi(\mk)=\mk'$ and $\phi(\mh)=\mh'$.
\end{definition}

Notice that if $\mh\subset\mk\subset\mg$ is a positive triple, then so is the triple $\mh\oplus\ma\subset\mk\oplus\ma\subset\mg\oplus\ma$
for any compact Lie algebra $\ma$, but this modification does not change the example topologically or metrically.  Thus, we need only classify the triple-isomorphism classes of \emph{reduced} positive triples, defined as follows:
\begin{definition}
A pair $\mk\subset\mg$ is called \emph{reduced} if the two Lie algebras do not share a common nontrivial idea.  Similarly, a triple $\mh\subset\mk\subset\mg$ is called \emph{reduced} if the three Lie algebras do not share a common nontrivial ideal.
\end{definition}

Next,  in Table~\ref{positive_pairs} we review the list of pairs $\mk\subset\mg$ for which
$G/K$ has positive curvature.  This will serve as our list of candidates for the base pair and also for the fiber pair of a positive triple.  For each pair, we include all intermediate subalgebras between $\mk$ and $\mg$, which is helpful in clarifying the embedding of $\mk$ in $\mg$.  We use the notation ``$\Delta \mh$'' to denote a diagonal embedding of $\mh$ in $\mh\oplus\mh$.

\begin{table}[h!]
\caption{Positive curvature pairs $G/K$}\label{positive_pairs}
\begin{tabular}{ r | c |  l }\hline
   ~ & $G/K$ & $\mk\subset \{\text{any intermediate subalgebras}\}\subset\mg$  \\
   \hline
   (1)  & $S^{n}$ & $so(n)\subset so(n+1)$ \\
   (2) & $S^{2n+1}$ & $su(n)\subset u(n)\subset su(n+1)$ \\
   (3) & $S^{4n+3}$ & $sp(n)\subset u(1)\oplus sp(n)\subset sp(1)\oplus sp(n)\subset sp(n+1)$ \\
   (4) & $S^{15}$ & $spin^+(7)\subset spin(8)\subset spin(9)$ \\
   (5) & $S^7$ & $\mg_2\subset spin(7)$ \\
   (6) & $S^6$ & $su(3)\subset \mg_2$ \\
   (7) & $\CP^n$ & $u(n)\subset su(n+1)$ \\
   (8) & $\CP^{2n+1}$ & $sp(n)\oplus u(1)\subset sp(n)\oplus sp(1)\subset sp(n+1)$ \\
   (9) & $\HP^n$ & $sp(n)\oplus sp(1)\subset sp(n+1)$ \\
   (10) & $\mathbb{OP}^2$ & $spin(9)\subset f_4$ \\
   (11) & $B^7$ & $so(3)_{\text{max}}\subset so(5)$ \\
   (12) & $B^{13}$ & $sp(2)\oplus u(1)\subset u(4)\subset su(5)$ \\
   (13) & $W_{1,1}$ & $\Delta u(2)\subset u(2)\oplus so(3)\subset su(3)\oplus so(3)$ \\
   (14) & $S^{2n+1}$ & $u(n)\subset u(1)\oplus u(n)\subset u(n+1)$ \\
   (15) & $S^{4n+3}$ & $sp(n)\oplus \Delta u(1)\subset sp(n)\oplus u(1)\oplus u(1)$ \\
        &  &  \hspace{1in}$\subset sp(n)\oplus sp(1)\oplus u(1)\subset sp(n+1)\oplus u(1)$ \\
   (16) & $S^{4n+3}$ & $sp(n)\oplus \Delta sp(1)\subset sp(n)\oplus sp(1)\oplus sp(1)\subset sp(n+1)\oplus sp(1)$ \\
\hline
\end{tabular} \end{table}

\begin{prop}\label{list} Any pair of compact Lie algebras $\mk\subsetneq\mg$
with the property that no two linearly independent vectors in $\mp=\mg\ominus\mk$ commute is pair-isomorphic to a pair of the form $\mk'\oplus\ma\subset\mg'\oplus\ma$, where $\mk'\subset\mg'$ is one of the 16 pairs from Table~\ref{positive_pairs}, and $\ma$ is a compact Lie algebra.
\end{prop}

This proposition follows from Berger's classification~\cite{Berger}, as corrected in~\cite{WP}. Notice that the pair $\mk\subset\mg$ is reduced if and only if $\ma=\{0\}$.

For all parameterized pairs in Table~\ref{positive_pairs}, the parameter $n$ can be taken from the range $n\geq 1$.  However, to avoid redundancy, it is better to exclude the $n=1$ case of pairs (7) and (9), each of which is pair-isomorphic to a different pair on the list.

One can find in~\cite{WP} descriptions of the embeddings $\mk$ in $\mg$ for the pairs (11), (12) and (13).  The embeddings for all other items on the list come from the well known classification of transitive effective actions on spheres \cite{Mo-Sa} and projective spaces \cite{Bor 1, Bor 2}.

In each pair $\mk\subset\mg$ from Table~\ref{positive_pairs}, the embedding of $\mk$ in $\mg$ is unique, up to conjugation, among embeddings within the same pair-isomorphism class, except for the following:
In case (1) when $n=7$, there are three non-conjugate embeddings of $so(7)\subset so(8)$, due to the Triality Principle.
In case (14) there is an infinite family of non-conjugate embeddings, parameterized by $k,l\in\Z$ (as in the fiber pair of Triple~\ref{T5p}).
In case (15) there is a one-parameter family of non-conjugate embeddings, given by the angle of the circle factor.  The fact that there are no others follows from Dynkin \cite[Theorem~15.1]{dynkin1} when $\mg$ is simple, and when $\mg$ is not simple via straightforward arguments that reduce to the simple case.

It is not quite true that every positive triple is one of the ones described in Sections 2 and Theorem~\ref{T:E}.  Rather, there are some additional ones which are obtained from these via the following construction:

\begin{lem}\label{L:Cheeger}
Suppose that $H\subset K\subset G$ is a positive triple.  Assume that there exists an intermediate subgroup, $J$,  between $H$ and $K$ (possibly equal to $K$) which is finitely covered by $H\times A$ for some nontrivial compact Lie group $A$.  Let $\rho:H\times A\ra J$ denote this $n$-fold covering homomorphism.  Assume that $\rho(H\times\{I\})=H$.  Define
$$B=\{(\rho(h,a^l),a^k)\mid h\in H, a\in A\}\subset J\times A\subset K\times A,$$
where $k,l$ are nonzero integers.  If $A$ is not isomorphic to $U(1)$, then assume that $n=k=l=1$.
Then the triple
$$B\subset K\times A\subset G\times A$$
is also positive.
The total space of this new triple, $(G\times A)/B$, is finitely covered by the total space of the old triple, $G/H$.  If $n=k=l=1$, then these total spaces are diffeomorphic.
\end{lem}
\begin{proof}
At the Lie algebra level, the new triple has the form $$\mh\oplus\Delta\ma\subset\mk\oplus\ma\subset\mg\oplus\ma.$$
Notice that $\mp\subset\mg$ is the same for the old and new triples.  Let $\mm_0=\mk\ominus\mh$ and $\mm_1=(\mk\oplus\ma)\ominus(\mh\oplus\Delta\ma)$ denote the
fiber components for the old and new triples respectively. Notice that the projection $\pi:\mg\oplus\ma\ra\mg$ sends $\mm_1$ bijectively onto $\mm_0$.  Further, if $X,Y\in\mm_1\oplus\mp$, then $\pi([X,Y])=[\pi(X),\pi(Y)]$.  From this, it is straightforward to see that any $A\in\mp$ that satisfies property ($iii$) of the definition of ``positive triple'' with respect to the old triple will also satisfy property ($iii$) with respect to the new triple.

Now suppose that $n=l=k=1$.  Using the fact that $\rho(\{I\}\times A)$ commutes with $\rho(H\times\{I\})$ inside $K$, it is straightforward to check that the function $f:(G\times A)/B\ra G/H$
defined as
$f([g,a]) = [g\cdot \rho(I, a^{-1})]$
is a well-defined diffeomorphism.

Finally, suppose that $A\cong U(1)$, and $k,l$ are arbitrary nonzero integers.  Notice that the natural action of $A$ on $M=G/H$ (defined as $a\star[g]=[\rho(I,a)\cdot g]$) is \emph{almost} free with ineffective kernel equal to the cyclic group $\Z_n\subset A$.  This is because $\rho(\{I\}\times A)$ and $\rho(H\times\{I\})$ commute inside $K$, and because $\{a\in A\mid \rho(I,a)\in\rho(H\times\{I\})\}=\Z_n$.  Consider the finite cyclic group $\Z_k\subset A$, and the quotient
$$M_k=M/\Z_k = G/\{\rho(h,\xi)\mid h\in H, \xi\in\Z_k\}.$$
Notice that $M_k$ is a smooth finite quotient of $M=G/H$ because $\Z_k$ acts almost freely on $M$ with ineffective kernel $\Z_k\cap\Z_n$.
Define the function $f:(G\times A)/B\ra M_k$ as
$$f([g,a]) = [g\cdot \rho(I, a^{-l/k})],$$
where $a^{-l/k}$ denotes $l^\text{th}$ power of the inverse of any $k^{\text{th}}$ root of $a$.  The result is independent of the choice of root.  It is straightforward to check that $f$ is a well-defined diffeomorphism.
\end{proof}

\begin{remark} We can apply the Lemma above to the following positive triples, with $A\cong U(1)$:
\begin{gather*}
\{I\}=SU(1)\subset U(1)\subset SU(2), \\
SU(n)\subset U(n)\subset SU(n+1),\\
Sp(n)\subset U(1)\oplus Sp(n)\subset Sp(n+1).
\end{gather*}
In each case, $G/H$ is a positively curved sphere found in Table~\ref{positive_pairs}, so the construction produces lens spaces.  On the other hand, applying the Lemma to the triple
$SU(2)\subset SU(3)\subset G_2$
provides an alternative proof of Corollary~\ref{zk}.
\end{remark}
\section{The classification of positive triples}
In this section, we prove the following classification theorem:
\begin{theorem}\label{T:Main}
Every reduced positive triple of Lie algebras is triple-isomorphic to one of the following:
\begin{enumerate}
\item A triple $\mh\subset \mk\subset \mg$ for which $G/H$ has positive curvature with respect to the normal homogeneous metric.
\item One of the previously known triples in Section 2.
\item One of the triples listed in Theorem~\ref{T:E}.
\item A triple derived from one of the possibilities above, by the method of Lemma~\ref{L:Cheeger}.
\end{enumerate}
\end{theorem}

Henceforth, assume that $\mh\subset\mk\subset\mg$ is a reduced positive triple, and that $A\in\mg$ is a particular vector satisfying property ($iii$) of the definition of {\sl positive triple}.  Since $A^\mp$ also satisfies this same property, we can assume without loss of generality that $A\in\mp$.
\begin{lem}\label{L:fat} The vector $A$ is strongly fat.
\end{lem}
\begin{proof}
If $A$ commutes with $Z\in\mm\oplus\mp$, and $A$ and $Z$ are linearly independent, then applying property ($iii$) in the definition of {\sl positive triple}  with $W=A$  gives
$[Z^\mm,[A,A]^\mk]\neq 0$, which is  impossible.
\end{proof}
\begin{lem}\label{L:triv}
The isotropy representation of $K$ on $\mp$ does not have a trivial factor.
\end{lem}
\begin{proof}  If $W\in\mp$ lies in a trivial factor, $\mp_0$, then for any $Z\in\mm$, we have $[Z,W]=0$. Furthermore, for any $X$ in $\mp$, $[W,X]$ lies in $\mp$ (since $\mk \oplus \mp_0$ is a subalgebra). In particular, $[A,W] \in \mp$. Hence, $[Z^\mm,[A,W]^\mk] = [Z^\mm,0] = 0$.
\end{proof}

\begin{lem}\label{L:fat2}Each irreducible component $\mp_i\subset\mp$ (of the isotropy representation of $K$ on $\mp$) contains a strongly fat vector.
\end{lem}
\begin{proof}
We choose a strongly fat vector $W\in\mp_i$ as follows.  If the projection of $A$ onto $\mp_i$ is nonzero, then choose $W$ to be that projection; otherwise, choose $W$ to be any nonzero vector in $\mp_i$.  In either case, notice that $[W,A]\in\mp$.  (To see this, write $A = A^i + A^\perp$ with $A^i \in \mp_i$ and $A^\perp\in\mp\ominus\mp_i$. Then $[W,A] = [W,A^\perp]$.  For any $Z \in \mk$, we have $0 = \lb W, [Z, A^\perp] \rb = \lb [A^\perp, W], Z \rb$. That is, $[A,W] \perp \mk$.)
Thus, if $W$ commutes with any $Z\in\mm\oplus\mp$, and $W$, $Z$ are linearly independent, then this would contradict property ($iii$).
\end{proof}

\begin{lem}\label{L:dim} For each irreducible component $\mp_i\subset\mp$ (of the isotropy representation of $K$ on $\mp$), we have $\text{dim}(\mm)<\dim(\mp_i)$.  In particular, $\dim(F)<\dim(B)$.
\end{lem}

\begin{proof}
By the previous Lemma, there exists a fat vector $W\in\mp_i$.  Since $W$ is fat, the kernel of the linear map $\text{ad}_W:\mk\ra\mp_i$ has a trivial intersection with $\mm$, so
$$\text{dim}(\mm)\leq\text{dim}(\mk)-\text{dim}(\text{Ker}(\text{ad}_W)) = \text{dim}(\text{Im}(\text{ad}_W))\leq \text{dim}(\mp_i)-1.$$
\end{proof}

\begin{lem}\label{L:dimsplit}
If $\mp$ splits as $\mp=\mp_1\oplus\mp_2$ such that $\ml=\mk\oplus\mp_1$ is an intermediate subalgebra between $\mk$ and $\mg$, then $\dimm(\mm)<\dimm(\mp_2)-\dimm(\mp_1)$.
\end{lem}
\begin{proof} By Lemma~\ref{L:fat2}, there exists a strongly fat vector $W\in\mp_2$.  Let $\mf$ denote the kernel of $\ad_W:\ml\ra\mp_2\ominus\text{span}\{W\}$. We have
$$\dimm(\mf)=\dimm(\ml)-\dimm(\text{im}(\ad_W))\geq \dimm(\ml)-(\dimm(\mp_2)-1).$$
Since $W$ is strongly fat, $\mf$ has a trivial intersection in $\ml$ with $\mm\oplus\mp_1$, which gives 
$$\dimm(\mf)\leq \dimm(\ml)-\dimm(\mm\oplus\mp_1)=\dimm(\ml)-\dimm(\mm)-\dimm(\mp_1).$$
Combining the above upper and lower bound on $\dimm(\mf)$ yields the result.
\end{proof}

For each pair from Table~\ref{positive_pairs} with no intermediate subalgebras, the isotropy representation is irreducible. For pairs with intermediate subalgebras, these subalgebras determine the splitting of $\mp$ into isotropy irreducible components.

Returning to our problem, we have a triple $\mh\subset\mk\subset\mg$ such that the \emph{fiber pair} $\mh\subset\mk$ and the \emph{base pair} $\mk\subset\mg$ both satisfy the conclusion of Proposition~\ref{list}, so each pair matches a pair from Table~\ref{positive_pairs}, whose number (1-16) will be referred to as the \emph{fiber type} and \emph{base type} respectively.

In cases where there are multiple possible embeddings, notice that without loss of generality, we can assume that $\mk\subset\mg$ is the standard embedding, but we cannot simultaneously assume that $\mh\subset\mk$ is the standard embedding.

\begin{lem}\label{L:nobase} The base type is not 2, 3, 4, 13, 14, or 15.  If the base type is 8, then $\text{dim}(\mm)=1$.  If the base type is 12 or 16, then $\text{dim}(\mm)\in\{1,2\}$.
\end{lem}
\begin{proof}  All of these claims come from the above lemmas together with the intermediate subalgebra information of Table~\ref{positive_pairs}.  First, if the base type is 2, 3, 14 or 15, then there is an intermediate subalgebra between $\mk$ and $\mg$ of the form $\mk\oplus\ma$, so $\ma$ is a trivial isotropy component, contradicting Lemma~\ref{L:triv}.

If the base type is 4, then $\mp=\mp_1\oplus\mp_2$, where $\mk \oplus \mp_1 = spin(8)$ is an intermediate subalgebra between $spin(7)$ and $spin(9)$.  Notice that $\text{dim}(\mp_1)=7$ and  $\text{dim}(\mp_2)=8$.  Lemma~\ref{L:dimsplit} implies that $\dimm(\mm)<8-7=1$, contradicting the definition of positive triple.  Similarly, if the base type is 13, then $\mp=\mp_1\oplus\mp_2$, where $\mk\oplus\mp_1=u(2)\oplus so(3)$ is an intermediate subalgebra between $\mk=u(2)$ and $\mg=su(3)\oplus so(3)$.  Notice that $\dim(\mp_1)=3$ and $\dim(\mp_2)=4$, so Lemma~\ref{L:dimsplit} implies that $\dimm(\mm)<4-3=1$, contradicting the definition of a positive triple.

In base type (8), we have an intermediate subalgebra whose dimension is $2$ more than the dimension of $\mk$, so $\text{dim}(\mm)<\text{dim}(\mp_1)=2$ by Lemma~\ref{L:dim}.  In base type 16 there is  an intermediate subalgebra whose dimension is 3 more than the dimension of $\mk$, so  $\text{dim}(\mm)<\text{dim}(\mp_1)=3$ by Lemma~\ref{L:dim}.

If the base type is 12, then $\mp=\mp_1\oplus\mp_2$, where $\mk\oplus\mp_1=u(4)$ is an intermediate subalgebra between $\mk=sp(2)\oplus u(1)$ and $\mg=su(5)$.  Notice that $\dim(\mp_1)=5$ and $\dim(\mp_2)=8$, so $\dim(\mm)<8-5=3$ by Lemma~\ref{L:dimsplit}.

\end{proof}

We divide the proof of Theorem~\ref{T:Main} into four Lemmas according to whether the base and/or fiber pair is reduced.

\begin{lem}\label{allredu}
Theorem~\ref{T:Main} is true if the base pair is reduced.
\end{lem}

\begin{proof}
For each possible type of the base pair $\mk\subset\mg$ that is allowed by Lemma~\ref{L:nobase} (with the standard embedding), we will consider each possibility for $\mh$ so that the fiber pair $\mh\subset\mk$ is also on the list (or is not reduced).   In cases where there are multiple possible embeddings of the fiber pair, we must consider all of them.   The cases below are enumerated first according to the base type.  We will write (Bx) and (Fy) to label respectively the cases with base type x and with fiber type $y$.
\begin{enumerate}
\item[(B1)] The triple has the form $\mh\subset so(n)\subset so(n+1)$.  The only compatible possibilities for the fiber pair are
    \begin{enumerate}
       \item[(F1)] $so(n-1)\subset so(n)\subset so(n+1)$ is a positive triple.  Here $M=T^1S^n$. When $n=8$, the alternative fiber embeddings yield the positive triple $spin^\pm(7)\subset spin(8)\subset spin(9)$.  Here $M=S^{15}$ considered as an $S^7$-bundle over $S^8$.
       \item[(F4)] $spin(7)\subset spin(9)\subset spin(10)$.  This $S^{15}$-bundle over $S^9$ is ruled out by Lemma~\ref{L:dim}.
       \item[(F5)] $\mg_2\subset spin(7)\subset spin(8)$.  This $S^7$-bundle over $S^7$ is ruled out by Lemma~\ref{L:dim}.
       \item[(F11)] $so(3)_\text{max}\subset so(5)\subset so(6)$.  This $B^7$-bundle over $S^5$ is ruled out by Lemma~\ref{L:dim}.
       \item[(n=3)] $\{0\}\subset so(3)\subset so(4)$.  This $S^3$-bundle over $S^3$ is ruled out by Lemma~\ref{L:dim}.
       \item[(n=4)] $sp(1)\oplus\{0\}\subset sp(1)\oplus sp(1)\cong so(4)\subset so(5)$ is triple-isomorphic to the positive triple $sp(1)\subset sp(1)\oplus sp(1)\subset sp(2)$, which gives $M=S^7$ considered as an $S^3$-bundle over $S^4$.
       \item[(n=5)] $\mh\subset sp(2)\cong so(5)\subset so(6)$ where $\mh=sp(1)$ or $\mh=sp(1)\oplus u(1)$ or $\mh=sp(1)\oplus sp(1)$.  These are bundles over $S^5$ (with fiber $S^7$ or $\CP^3$ or $\HP^1=S^4$ respectively).  The first two are ruled out by Lemma~\ref{L:dim}.  The final is triple-isomorphic to $so(4)\subset so(5)\subset so(6)$, which is a previously mentioned positive triple (F1).
       \item[(n=6)] $\mh\subset su(4)\cong so(6)\subset so(7)$ where $h=su(3)$ or $h=u(3)$.  This is a bundle over $S^6$ with fiber $S^7$ or $\CP^3$ respectively.   Both are ruled out by Lemma~\ref{L:dim}.
     \end{enumerate}
\item[(B5)] $su(3) \subset \mg_2 \subset spin(7)$ is a positive triple, as proven in Section 3.  Here $M=T^1S^7$.
\item[(B6)] $\mh\subset su(3)\subset\mg_2$, where $\mh\in\{su(2),u(2)\}$.  Both possibilities for $\mh$ yield positive triples, as proven in Section 3.  Here $M=T^1S^6$ or $M$ is a $\CP^2$-bundle over $S^6$.
\item[(B7)] The triple has the form $\mh\subset su(n)\oplus u(1)\cong u(n)\subset su(n+1)$.  The only possibility with a reduced fiber pair is
  \begin{enumerate}
  \item[(F14)] $u(n-1)\subset u(n)\subset su(n+1)$.  Here there are infinitely many different embeddings of $\mh$ in $\mk$, as described in Equation~(\ref{T5p}), all of which give positive triples.
  \end{enumerate}
  The only possibility for which $\mh$ and $\mk$ share $su(n)$ as a common ideal is
  \begin{enumerate}
  \item[(F1)] $su(n)\subset u(n)\subset su(n+1)$ is a positive triple.  Notice that the fiber pair reduces to $so(1)\subset so(2)$.  Here $M=S^{2n+1}$ considered as an $S^1$-bundle over $\CP^n$.
  \end{enumerate}
  Finally, the only possibilities for which $\mh$ and $\mk$ share $u(1)$ as a common ideal are
  \begin{enumerate}
  \item[(F2)] $su(n-1)\oplus u(1)\subset su(n)\oplus u(1) \cong u(n)\subset su(n+1)$ is positive Triple~\ref{T5p} with $k=-n$ and $l=1$.
  \item[(F7)] $u(n-1)\oplus u(1)\subset su(n)\oplus u(1) \cong u(n)\subset su(n+1)$ is positive Triple~\ref{green}.  Here $M$ is the projective tangent bundle of $\CP^n$.
  \item[(F12)] $sp(2)\oplus u(1)\oplus u(1)\subset su(5)\oplus u(1)\cong u(5)\subset su(6)$.  This $B^{13}$-bundle over $\CP^5$ is ruled out by Lemma~\ref{L:dim}.
   \item[(n=4)] $so(5)\oplus u(1)\subset so(6)\oplus u(1)\cong su(4)\oplus u(1)\cong u(4)\subset su(5)$ is a positive triple.  Here, $M=B^{13}$ is viewed as an $\RP^5$-bundle over $\CP^4$.
   \end{enumerate}
\item [(B8)] $\mh\subset sp(n)\oplus u(1)\subset sp(n+1)$.  Lemma~\ref{L:nobase} implies that $\text{dim}(\mm) = 1$.  The only possibility is
$sp(n)\subset sp(n)\oplus u(1)\subset sp(n+1),$ which is a positive triple.  Here $M=S^{4n+3}$ is viewed as an $S^1$-bundle over $\CP^{2n+1}$.
\item[(B9)] The triple has the form $\mh\subset sp(n)\oplus sp(1)\subset sp(n+1)$ with $n>1$.  The only possibility with a reduced fiber pair is
    \begin{enumerate}
    \item [(F16)] $sp(n-1)\oplus \Delta sp(1)\subset sp(n)\oplus sp(1)\subset sp(n+1)$ is the positive triple for the unit tangent bundle of $\HP^n$.
    \end{enumerate}
    The only possibilities for which $\mh$ and $\mk$ share $sp(n)$ as a common ideal are
    \begin{enumerate}
    \item[(F1)] $sp(n)\oplus u(1)\subset sp(n)\oplus sp(1)\subset sp(n+1)$ is a positive triple for which $M=\CP^{2n+1}$ viewed as an $S^2$-bundle over $\HP^n$.
    \item [(F2)] $sp(n)\subset sp(n)\oplus sp(1)\subset sp(n+1)$ is a positive triple for which $M=S^{4n+3}$ viewed as an $S^3$-bundle over $\HP^{n}$.
    \end{enumerate}
    Finally, the only possibilities for which $\mh$ and $\mk$ share $sp(1)$ as a common ideal are
    \begin{enumerate}
    \item [(F3)] $sp(n-1)\oplus sp(1)\subset sp(n)\oplus sp(1)\subset sp(n+1)$ is positive Triple~\ref{T4}.  Here, $M$ is an $S^{4n-1}$-bundle over $\HP^{n}$.
    \item [(F8)] $h'\oplus sp(1)\subset sp(n)\oplus sp(1)\subset sp(n+1)$, where $h'=sp(n-1)\oplus u(1)$, is positive Triple (4) from Theorem~\ref{T:E}.  Here, $M$ is a $\CP^{2n-1}$-bundle over $\HP^{n}$.
    \item [(F9)] $h'\oplus sp(1)\subset sp(n)\oplus sp(1)\subset sp(n+1)$, where $h'=sp(n-1)\oplus sp(1)$, is the positive triple for the projective tangent bundle of $\HP^n$.  Notice that when $n=2$, this is Wallach's flag manifold $W^{12}$.
    \item [(F11)] $so(3)_{\text{max}}\oplus sp(1)\subset so(5)\oplus sp(1)\cong sp(2)\oplus sp(1)\subset sp(3)$ is positive Triple (5) from Theorem~\ref{T:E}.
    \end{enumerate}
\item[(B10)] $\mh\subset spin(9)\subset f_4$ where $\mh\in\{spin(7),spin(8)\}$ are positive Triples~\ref{T2} and~\ref{T2P}.  Here $M=T^1\mathbb{OP}^2$ and $M=P_{\mathbb{O}}T\mathbb{OP}^2$ respectively.
\item[(B11)] $so(2)\subset so(3)_\text{max}\subset so(5)$ is \emph{not} a positive triple, due to the following argument.  We first claim that there exists a $4$-dimensional subspace $\mathcal{W}\subset\mp$ that does not contain any strongly fat vectors.  To see this, let $\mt^2 \subset \mg=so(5)$ be a maximal abelian subalgebra containing $\mh=so(2)$.  Consider the intermediate subalgebra $\mt^2\subset so(2)\oplus so(3)\subset\mg$. Define $\ml=\mg\ominus(so(2)\oplus so(3))$.  Since $SO(5)/(SO(2)\times SO(3))$ is a symmetric space of rank $>1$, each vector in $\ml$ commutes with at least one other vector in $\ml$ that is not a scalar multiple of itself.  Since $\ml\subset\mm\oplus\mp$, this implies that the space $\mathcal{W}=\ml\cap\mp$ does not contain any strongly fat vectors.  By a dimension count, $\mathcal{W}$ has dimension at least $4$.
    
    Now let $A\in\mp$.  We will prove that $A$ does not satisfy property ($iii$) of the definition of positive triple.  Let $\mf$ denote the kernel of the map from $\mp$ to $\mk$ defined as $W\mapsto[A,W]^\mk$.  The dimension of $\mf$ is at least $4$.  Since $\mf$ and $\mathcal{W}$ both have dimension at least $4$, they must intersect nontrivially in $\mp$. Choose a nonzero vector $W\in\mf\cap\mathcal{W}$.  Since $W$ is not strongly fat, there exists $Z\in\mm\oplus\mp$ such that $[W,Z]=0$. Notice that $[[A,W]^\mk,Z^\mm]=[0,Z^\mm]=0$, so $A$ does not satisfy property ($iii$).
\item[(B12)] $\mh\subset sp(2)\oplus u(1)\subset su(5)$ gives a bundle over the Berger space $B^{13}$.  By Lemma~\ref{L:nobase}, $\dim(\mm)\in\{1,2\}$.  The only possibility is $\mh=sp(2)$.  This $S^1$-bundle over $B^{13}$ can be ruled out as follows.  Notice that $\mm$ is the center of $\mk$, and the embedding is such that $\mm$ commutes with $\mp_1$.  For any nonzero $Z\in\mm$, $W\in\mp_1$ and $A\in\mp$, we have $[Z,W]=0$ and $[Z,[A,W]^\mk] =0$ because $Z$ lies in the center of $\mk$. Thus condition ($iii$) cannot hold.
\item[(B16)] $\mh\subset sp(n)\oplus \Delta sp(1)\subset sp(n+1)\oplus sp(1)$.  Lemma~\ref{L:nobase} implies that
$\dim(\mm)\in\{1,2\}$.  The only possibility is $\mh=sp(n)\oplus u(1)$, where $u(1) \subset \Delta sp(1)$, which gives an $S^2$-bundle over $S^{4n+1}$.  We now prove that this triple is not positive.  Notice that $\mp=\mp'\oplus\mp_1$, where $\mp_1=sp(n+1)\ominus(sp(n)\oplus sp(1))$ and $\mp'=(sp(1)\oplus sp(1))\ominus\Delta sp(1)$. If $q\in sp(1)$ and $w\in\HH^n\cong\mp_1$, we will use the shorthand ``$(q,-q,w)$'' to denote the element $(q,-q)+w\in\mp'\oplus\mp_1=\mp$.  Denoting $sp(1)=\text{span}\{i,j,k\}$, we can assume without loss of generality that $\mh=\text{span}\{(i,i,0)\}$, so $\mm=\{(q,q,0)\mid q\in\text{span}\{j,k\}\}$.

Let $q\in sp(1)$ and $w\in\HH^n$, so that $A=(q,-q,w)$ represents an arbitrary element of $\mp$.  We will show that $A$ does not satisfy property ($iii$) of the definition of positive triple.  For any $p\in\text{span}\{j,k\}$ and any $v\in\HH^n$, notice that $W=(0,0,v)\in\mp$ will commute with $Z=(p,0,v)\in\mm\oplus\mp$.  Notice that $Z^\mm=\frac 12 (p,p,0)$ and $[A,W]^k=[(0,0,w),(0,0,v)]^\mk$.  Thus, $[Z^\mm,[A,W]^\mk]=0$ if and only if $p$ is parallel to the projection of $[w,v]$ onto $sp(1)$, where the Lie bracket $[w,v]$ is being taken in $sp(n+1)$.  We first choose $v$ to be orthogonal to $i\cdot w$, which ensures that the projection of $[w,v]$ onto $sp(1)$ is orthogonal to $i$.  We then can choose $p$ to equal this projection, thus ensuring that $[Z^\mm,[A,W]^\mk]=0$.

\end{enumerate}
\end{proof}

\begin{lem}
Theorem~\ref{T:Main} is true if the fiber pair is reduced.
\end{lem}

\begin{proof}
By the previous Lemma, We can assume that the base pair is not reduced, which means that the triple has the form $$\mh\subset\mk'\oplus\ma\subset\mg'\oplus\ma,$$ where $\ma$ is any compact Lie algebra, and $\mk'\subset\mg'$ is one of the pairs (1) through (16) from Table~\ref{positive_pairs} allowed by Lemma~\ref{L:nobase}.
The fiber pair, $\mh\subset\mk'\oplus\ma$, must be one of the pairs of Table~\ref{positive_pairs} for which the larger algebra is a product: pair (1) when $n=3$, or (13), (14), (15), or (16).  We will now consider each of these possibilities.  The cases below are numbered according to the fiber pair.

\begin{itemize}
\item[(F1)]
 If the fiber pair is the $n=3$ case of (1), then the triple is one of
     \begin{eqnarray*}
     so(3)\subset so(4)\cong so(3)\oplus so(3) & \subset & so(4)\oplus so(3),\\
     so(3)\subset so(4)\cong so(3)_{max} \oplus so(3) & \subset & so(5)\oplus so(3).
     \end{eqnarray*}
The first possibility gives an $S^3$-bundle over $S^3$, which is ruled out by Lemma~\ref{L:dim}.  The second possibility, over $B^7$,  is ruled out as follows.  Every nonzero $X\in\mm$ has a nonzero projection onto $so(3)_{max}$, and therefore commutes with a unique vector in $\mp=so(5)\ominus so(3)_{max}$, so there could not be a fat vector.
\item[(F13)]
When the fiber is $W_{1,1}$, the possibilities for the triple are
\begin{eqnarray*}
u(2)\subset su(3)\oplus so(3) & \subset & \mg_2\oplus so(3), \\
u(2)\subset su(3)\oplus so(3) & \subset & su(3)\oplus so(4), \\
u(2)\subset su(3)\oplus so(3)_{\max} & \subset & su(3)\oplus so(5).
\end{eqnarray*}
These correspond to bundles with fiber $W_{1,1}$ and with base spaces  $S^6$, $S^3$, and $B^7$  respectively, so they are all ruled out by Lemma~\ref{L:dim}.
\item[(F14)]
The fiber pair is $u(n-1)\subset u(n)$ with an \emph{arbitrary} embedding.  The possibilities for the triple are
\begin{eqnarray*}
u(n-1)\subset u(n)\cong su(n)\oplus u(1) & \subset & su(n)\oplus sp(1), \\
u(3)\subset u(4)\cong su(4)\oplus u(1) & \cong & so(6)\oplus u(1)\subset so(7)\oplus u(1),\\
u(2)\subset u(3)\cong su(3)\oplus u(1) & \subset & \mg_2\oplus u(1),\\
u(1)\subset u(2)\cong so(3)\oplus u(1) & \subset & so(4)\oplus u(1),\\
u(1)\subset u(2)\cong so(3)_{max}\oplus u(1) & \subset & so(5)\oplus u(1).
\end{eqnarray*}
The first, second, and fourth possibilities are ruled out by Lemma~\ref{L:dim}.  The third possibility is positive because it is obtained by the method of Lemma~\ref{L:Cheeger} from the positive triple $su(2)\subset su(3)\subset \mg_2$.  The final possibility is rulled out as follows.  Since every vector in $\mm$ has a nonzero projection onto $so(3)_{max}$, if this triple were positive, then the triple $u(1)\subset so(3)_{max}\subset so(5)$ would also be positive, contradicting our previous proof that it is not.
\item[(F15)]
The fiber pair is $sp(n-1)\oplus u(1)\subset sp(n)\oplus u(1)$ with an \emph{arbitrary} embedding.  The possibilities for the triple are
 \begin{eqnarray*}
    sp(1)\oplus u(1)\subset sp(2)\oplus u(1) & \cong & so(5)\oplus u(1)\subset so(6)\oplus u(1),\\
    sp(n-1)\oplus u(1) \subset sp(n)\oplus u(1) & \subset & sp(n)\oplus sp(1).
\end{eqnarray*}
Both of these possibilities are ruled out by Lemma~\ref{L:dim}.
\item[(F16)]
The possibilities for the triple are
 \begin{eqnarray*}
    sp(1)\oplus sp(1) \subset sp(2)\oplus sp(1) & \cong & so(5)\oplus sp(1)\subset so(6)\oplus sp(1),\\
    sp(n-1)\oplus sp(1) \subset sp(n)\oplus sp(1) & \cong & sp(n)\oplus so(3)\subset  sp(n)\oplus so(4), \\
    sp(n-1)\oplus sp(1) \subset sp(n)\oplus sp(1) & \cong  & sp(n)\oplus so(3)_{max}  \subset  sp(n)\oplus so(5).
\end{eqnarray*}
All of these possibilities are all ruled out by Lemma~\ref{L:dim}.
\end{itemize}
\end{proof}

It remains to consider the possibility that neither the fiber pair nor the base pair is reduced.  We divide this remaining case into two lemmas.  The first has an added hypothesis that prevents the fiber and base pair embeddings from twisting together in $\mk$.

\begin{lem}
Theorem~\ref{T:Main} is true if neither the fiber pair nor the base pair is reduced and the dimension of the center of $\mk$ is $0$ or $1$.
\end{lem}

\begin{proof}
The triple has the form
\begin{equation}\label{HP1}
\mh'\oplus\ml\subset\tilde{\mk}''\oplus\ml\cong\tilde{\mk}'\oplus\ma\subset\mg'\oplus\ma,
\end{equation}
where $\ml$ and $\ma$ are compact Lie algebras, and where the pairs $\mh'\subset\tilde{\mk}''$ and $\tilde{\mk}'\subset\mg'$ come from Table~\ref{positive_pairs}, (1) through (16).

Since the triple is reduced, we know that $\ma$ and $\ml$ intersect trivially in $\mk\cong\tilde{\mk}''\oplus\ml\cong\tilde{\mk}'\oplus\ma$.  Consider the decomposition of $\mk$ into a sum of simple ideals plus a Euclidean factor, $\R^l$, which is the center of $\mk$.  By hypothesis, we know that $l\leq 1$, so $\mk$ decomposes uniquely into a sum of ideals each of which is simple or is $\R^1$.  Since $\ma$ and $\ml$ are ideals of $\mk$ that intersect trivially in $\mk$, each must be a sum of a subcollection of these ideals into which $\mk$ decomposes, and the two subcollections must be disjoint.  This implies that the triple has the form
\begin{equation}\label{HP2}\mh'\oplus\ml\subset\mk'\oplus\ml\oplus\ma\subset\mg'\oplus\ma,
\end{equation}
where $\ml$ and $\ma$ are compact Lie algebras, and where the pairs $\mh'\subset\mk'\oplus\ma$ and $\mk'\oplus\ml\subset\mg'$ come from Table~\ref{positive_pairs}, (1) through (16).  If $\mk'=\{0\}$, then the triple could not be positive, since in this case $\mp$ would lie in the first factor of $\mg=\mg'\oplus\ma$ (that is, $\mp\subset\mg'$) while $\mm$ would lie in the second factor (that is, $\mm\subset\ma$).  We therefore assume that $\mk'$ is nontrivial.

The fiber pair, $\mh'\subset\mk'\oplus\ma$, is one of the pairs on the list for which the larger Lie algebra is a product: case (1) when $n=3$  or one of the pairs (13), (14), (15),  (16).  Notice that $\text{dim}(F)\geq 3$ for each of these possibilities.  The base pair, $\mk'\oplus\ml\subset\mg'$, is one of the pairs on the list for which the smaller Lie algebra is a product.  The only such pairs allowed by Lemma~\ref{L:nobase} are: case (1) when $n=4$ or one of the pairs (7), (9).

We now consider each of these possibilities for the fiber pair, and for each one, we consider each compatible possibility for the base pair.

\begin{enumerate}
\item[(F1)]
The triple has the form
$$\Delta so(3)\oplus\ml \subset so(3)\oplus so(3)\oplus \ml\subset so(3)\oplus \mg',$$
so the base pair has the form $so(3)\oplus\ml\subset\mg'$.  The possibilities are
\begin{itemize}
\item[(B7)] $\Delta so(3) \oplus u(1) \subset so(3)\oplus so(3)\oplus u(1)\cong so(3)\oplus u(2)\subset so(3)\oplus su(3)$ is the positive triple for $G/H=W_{1,1}$ considered as an $S^3$-bundle over $\CP^2$.
\item[(B1)] $\Delta so(3) \oplus so(3) \subset (so(3) \oplus so(3)) \oplus so(3) \cong so(3) \oplus (so(3) \oplus so(3)) \subset so(3) \oplus so(5)$.  This is the $m=1$ case of the following triple:
\item[(B9)]  $\Delta so(3) \oplus sp(m) \subset (so(3) \oplus so(3)) \oplus sp(m) \cong so(3) \oplus (sp(1) \oplus sp(m)) \subset so(3) \oplus sp(m+1)$.  This is a positive triple because $G/H$ has positive curvature.  In fact, $G/H=S^{4m+3}$ considered as an $S^3$-bundle over over $\HP^m$, and the pair $\mh\subset\mg$ is of type (16).
\end{itemize}

\item[(F13)] The triple has the form
$$u(2)\oplus\ml\subset su(3)\oplus so(3)\oplus \ml\subset \mg'\oplus\ma,$$
where $\ma\in\{su(3),so(3)\}$, so the base pair has the form $so(3)\oplus\ml\subset\mg'$ or $su(3)\oplus\ml\subset\mg'$.
Here the fiber is 7-dimensional, so Lemma~\ref{L:dim} rules out base types (1) and (7).  The only remaining compatible possibility is
\begin{itemize}
\item[(B9)] $u(2)\oplus sp(m) \subset su(3) \oplus so(3)\oplus sp(m)\cong su(3) \oplus (sp(1) \oplus sp(m))
\subset su(3) \oplus sp(m+1)$.  We know that  $\mp$ lies in second ideal of $\mg$ ($\mp \subset sp(m+1)$). Further, dimension counting establishes that $\mm$ has a nontrivial intersection with the first ideal of $\mg$
($\mm\cap su(3)\neq\{0\}$).  Therefore, $\mp$ does not contain any fat vectors.
\end{itemize}

\item[(F14)]  The triple has the form
$$u(n)\oplus\ml \subset u(n+1)\oplus\ml\cong su(n+1)\oplus u(1)\oplus\ml\subset\mg'\oplus\ma,$$
where $n\geq 1$ and $\ma\in\{u(1),su(n+1)\}$, so the base pair has the form $su(n+1)\oplus\ml\subset\mg'$ or
$u(1)\oplus\ml\subset\mg'$.  In the fiber pair, there are infinitely many possible embedding of $u(n)$ in $u(n+1)$.
We first assume that $\ma=su(n+1)$, so the base pair has the form $u(1)\oplus\ml\subset\mg'$.  In this case, the only compatible possibility for the base type is
\begin{itemize}
\item[(B7)]  $u(n) \oplus su(m) \subset su(n+1) \oplus u(1) \oplus su(m)   \subset su(n+1) \oplus su(m+1)$. Notice, $\mp$ lies in the second ideal of $\mg$ ($\mp\subset su(m+1)$).  Further, there exists $X\in\mm$ such that $X$ lies in the first factor of $\mg$ ($X\in su(n+1)$) because $X$ is orthogonal to the $u(1)$-factor of $\mk$.  Therefore, $\mp$ cannot contain a strongly fat vector.
\end{itemize}

We next assume that $\ma = u(1)$, so the base pair has the form $su(n+1)\oplus\ml\subset\mg'$.  The possibilities are
\begin{itemize}
\item[(B7)]  $u(n) \oplus u(1) \subset (u(1) \oplus su(n+1)) \oplus u(1) \cong u(1) \oplus (su(n+1) \oplus u(1)) \subset u(1) \oplus su(n+2)$. This is a positive triple obtained from Triple~\ref{T5p} by the method of Lemma~\ref{L:Cheeger}.

\item[(B1)] $u(1) \oplus so(3) \subset (u(1) \oplus su(2)) \oplus so(3) \cong u(1) \oplus (so(3) \oplus so(3))
\subset u(1) \oplus so(5)$. This is the $m=1$ case of the following triple:

\item[(B9)] $u(1) \oplus sp(m) \subset u(2) \oplus sp(m) \cong u(1) \oplus (sp(1) \oplus sp(m)) \subset u(1) \oplus sp(m+1)$.  This is a positive triple because $G/H$ has positive curvature.  In fact, $G/H= S^{4n+1}$ considered as an $S^3$-bundle over $\HP^m$, and the pair $\mh\subset\mg$ is of type (15).
\end{itemize}

\item[(F15)] The triple has the form
$$sp(n)\oplus u(1)\oplus\ml\subset sp(n+1)\oplus u(1)\oplus l\subset\mg'\oplus\ma,$$
where $n\geq 1$ and $\ma\in\{sp(n+1),u(1)\}$, so the base pair has the form $u(1)\oplus\ml\subset\mg'$ or the form $sp(n+1)\oplus\ml\subset\mg'$.  The fiber is $S^{4n+3}$, which has dimension at least $7$.  The embedding of the fiber pair has an arbitrary slope.  We assume first that $\ma=u(1)$, so the base pair has the form $sp(n+1)\oplus\ml\subset\mg'$.  In this case, the only compatible possibility for the base type that is not ruled out by Lemma~\ref{L:dim} is the following:
\begin{itemize}
\item[(B9)] $sp(n)\oplus u(1) \oplus sp(1) \subset sp(n+1)\oplus u(1) \oplus sp(1) \subset  u(1)\oplus  sp(n+2)$.   This is a positive triple obtained from Triple~\ref{T4} by the method of Lemma~\ref{L:Cheeger}.
\end{itemize}
We next assume that $\ma=sp(n+1)$, so the base pair has the form $u(1)\oplus\ml\subset\mg'$.  In this case, the only compatible possibility for the base type that is not ruled out by Lemma~\ref{L:dim} is the following:
\begin{itemize}
\item[(B7)] $sp(n)\oplus u(1) \oplus su(m) \subset sp(n+1) \oplus u(1) \oplus su(m) \subset sp(n+1)\oplus su(m+1)$.
This triple has no strongly fat vector because $\mp$ is contained in the second ideal of $\mg$ ($\mp \subset su(m+1)$) while $\mm$ intersects nontrivially with the first ideal of $\mg$ ($\mm\cap sp(n+1)\neq\emptyset$).
\end{itemize}

\item[(F16)]  The triple has the form
$$sp(n) \oplus sp(1)\oplus\ml \subset sp(n+1) \oplus sp(1)\oplus\ml\subset\mg'\oplus\ma,$$
where $n\geq 1$ and $\ma\in\{sp(1),sp(n+1)\}$, so the base pair has the form either $sp(n+1)\oplus\ml\subset\mg'$ or  $sp(1)\oplus\ml\subset\mg'$.  The fiber is $S^{4n+3}$, which has dimension at least $7$.
We first assume that $\ma=sp(n+1)$, so the base pair has the form $sp(1)\oplus\ml\subset\mg'$.  In this case, the only compatible possibility for the base type that is not ruled out by Lemma \ref{L:dim} is
\begin{itemize}
\item[(B9)] $sp(n) \oplus sp(1) \oplus sp(m) \subset sp(n+1) \oplus sp(1) \oplus sp(m) \subset sp(n+1) \oplus sp(m+1)$.  This triple has no fat vector because $\mp$ is contained in the second ideal of $\mg$ ($\mp \subset sp(m+1)$) while $\mm$ intersects nontrivially with the first ideal of $\mg$ ($\mm\cap sp(n+1)\neq\emptyset)$.
\end{itemize}
Next, assume that $\ma=sp(1)$, so the base pair has the form $sp(n+1)\oplus\ml\subset\mg'$. In this case, all compatible possibilities for the base type are ruled out by Lemma \ref{L:dim} except 
\begin{itemize}
\item[(B9)] $sp(1) \oplus sp(n)\oplus sp(1) \subset sp(1)\oplus sp(n+1) \oplus sp(1) \subset  sp(1)\oplus  sp(n+2)$.   This is another positive triple obtained from Triple~\ref{T4} by the method of Lemma~\ref{L:Cheeger}.
\end{itemize}
\end{enumerate}
\end{proof}

\begin{lem}
Theorem~\ref{T:Main} is true if neither the fiber pair nor the base pair is reduced.
\end{lem}

\begin{proof}
As in the proof of the previous lemma, the triple has the form
\begin{equation}
\mh'\oplus\ml\subset\tilde{\mk}''\oplus\ml\cong\tilde{\mk}'\oplus\ma\subset\mg'\oplus\ma,
\end{equation}
where $\ml$ and $\ma$ are compact Lie algebras, and where the pairs $\mh'\subset\tilde{\mk}''$ and $\tilde{\mk}'\subset\mg'$ are from Table~\ref{positive_pairs}, (1) through (16).

Since the triple is reduced, we know that $\ma$ and $\ml$ intersect trivially in $\mk\cong\tilde{\mk}''\oplus\ml\cong\tilde{\mk}'\oplus\ma$.  If $\ml$ and $\ma$ are orthogonal in $\mk$, then the triple has the form of Equation~\ref{HP2}, and was therefore already considered in the previous lemma.  It remains to consider the possibility that $\ml$ and $\ma$ are not orthogonal, which is only possible when each of $\ml$, $\ma$, $\tilde{\mk}'$ and $\tilde{\mk}''$ has a nontrivial center.

The base pair $\tilde{\mk}'\subset\mg'$ is a pair from Table~\ref{positive_pairs} for which the smaller Lie algebra has a nontrivial center.  The only such pairs compatible with Lemma~\ref{L:nobase} are (1) with $n=2$, (7), (8), or (12).  For each of these possibilities, the center of $\tilde{\mk}'$ is one-dimensional.  Meanwhile, the fiber pair $\mh'\subset\tilde{\mk}''$ is a pair from Table~\ref{positive_pairs} for which the larger Lie algebra has a nontrivial center; namely (1) with $n=1$, (14), or (15).  For each of these possibilities, the center of $\tilde{\mk}''$ is one-dimensional.

Since $\tilde{\mk}'$ and $\tilde{\mk}''$ have one-dimensional centers, we know that the centers of $\ma$ and $\ml$ have the same dimension, and since the triple is reduced, their common dimension must be $1$.  Thus, we can write $\ma=\ma_0\oplus u(1)$ and $\ml=\ml_0\oplus u(1)$, where $\ma_0$ and $\ml_0$ are (possibly trivial) semisimple ideals.  Let $\mk'$ denote the orthogonal compliment of $\ma_0\oplus\ml_0$ inside of the semisimple part of $\mk$.  We now have $\mk$ decomposed as
$$\mk = \ma_0\oplus u(1)\oplus u(1)\oplus\ml_0\oplus \mk'.$$
We can assume without loss of generality that $\ma$ equals $\ma_0$ plus the first $u(1)$-factor, while $\ml$ equals $\ml_0$ plus a circle factor that is embedded diagonally into $u(1)\oplus u(1)$ with some nonzero slope.

We now have that our triple is of the following form:
$$\Delta u(1)\oplus\ml_0\oplus\mh' \subset\ma_0\oplus u(1)\oplus u(1)\oplus\ml_0\oplus \mk' \subset
  \ma_0\oplus u(1)\oplus \mg',$$
where $\Delta u(1)$ denotes an embedding of $u(1)$ in $u(1)\oplus u(1)$ with arbitrary slope.
Here, the fiber pair is $\mh'\subset \ma_0\oplus \Delta^\perp u(1)\oplus\mk'$, where $\Delta^\perp u(1)$ denotes the orthogonal compliment of $\Delta u(1)$ in $u(1)\oplus u(1)$.  For this fiber pair to match one of the pairs of Table~\ref{positive_pairs}, either $\ma_0$ or $\mk'$ must be trivial.  Meanwhile, the base pair is $u(1)\oplus \ml_0\oplus\mk'\subset\mg'$.  For this base pair to match one of the pairs of Table~\ref{positive_pairs}, either $\ml_0$ or $\mk'$ must be trivial.
In summary, at least one of $\{\mk',\ma_0\}$ is trivial and at least one of $\{\mk',\ml_0\}$ is trivial. 

Suppose first that $\mk'$ is trivial and $\ma_0$ is nontrivial.  In this case, $\mp$ lies in the $\mg'$-factor of $\mg=\ma_0\oplus u(1)\oplus\mg'$, while $\mm$ has only a $1$-dimensional projection onto this $\mg'$-factor.  In all of the previously listed possibilities for the fiber pair, $\mm$ contains vectors orthogonal to this $1$-dimensional projection.  These vectors commute with $\mp$, so the triple has no fat vectors and therefore could not be not positive.

Suppose next that both $\mk'$ and $\ma_0$ are trivial, so the triple has the form
\begin{equation}\label{squijim}\Delta u(1)\oplus\ml_0\subset u(1)\oplus u(1)\oplus\ml_0\subset u(1)\oplus \mg'.\end{equation}
There are several possibilities, depending on the type of the base pair, $u(1)\oplus\ml_0\subset\mg'$:
\begin{enumerate}
\item[(B1)] $\Delta u(1)\subset u(1)\oplus u(1)\cong u(1)\oplus so(2)\subset u(1)\oplus so(3),$
\item[(B7)] $\Delta u(1)\oplus su(n)\subset u(1)\oplus u(1)\oplus su(n)\subset u(1)\oplus su(n+1),$
\item[(B8)] $\Delta u(1)\oplus sp(n)\subset u(1)\oplus u(1)\oplus sp(n)\subset u(1)\oplus sp(n+1),$
\item[(B12)] $\Delta u(1)\oplus sp(2)\subset u(1)\oplus u(1)\oplus sp(2)\subset u(1)\oplus su(5),$
\end{enumerate}
Notice that Triple~\ref{squijim} is obtained from the following triple via the method of Lemma~\ref{L:Cheeger}:
\begin{equation}\label{squijim2}\ml_0\subset u(1)\oplus\ml_0\subset \mg'.\end{equation}
The first three triples above (base types (1), (7), and (8)) are positive triples because, in these cases, Triple~\ref{squijim2} is a positive triple.   However, the total spaces are all lens spaces, as explained at the end of Section 4.  The remaining triple above (base type (12)) is not positive, which is straightforward to conclude using that fact that, in this case, Triple~\ref{squijim2} is not a positive triple, as proven in Lemma~\ref{allredu}.

It remains to consider the case where $\mk'$ is nontrivial, in which case both $\ml_0$ and $\ma_0$ must be trivial, so the triple is of the form
$$\Delta u(1)\oplus\mh' \subset u(1)\oplus u(1)\oplus \mk' \subset
 u(1)\oplus \mg'.$$

Fiber type (1) is not compatible.  When the fiber type is (14), this becomes $$\Delta u(1)\oplus u(n) \subset u(1)\oplus u(1)\oplus su(n+1) \subset u(1)\oplus \mg'.$$  The compatible possibilities for the base type are
\begin{enumerate}
\item[(B7)] $\Delta u(1)\oplus u(n) \subset u(1)\oplus u(1)\oplus su(n+1) \subset u(1)\oplus su(n+2),$
\item[(B8)] $\Delta u(1)\oplus u(1) \subset u(1)\oplus u(1)\oplus su(2)\cong u(1)\oplus u(1)\oplus sp(1) \subset u(1)\oplus sp(2),$
\end{enumerate}
From this list, base type (7) is positive Triple~\ref{T5}, while base type (8) is ruled out by Lemma~\ref{L:nobase}.

When the fiber type is (15), this becomes
$$\Delta u(1)\oplus sp(n)\oplus u(1) \subset u(1)\oplus u(1)\oplus sp(n+1) \subset u(1)\oplus \mg'.$$  The compatible possibilities for the base type are
\begin{enumerate}
\item[(B8)] $\Delta u(1)\oplus sp(n)\oplus u(1) \subset u(1)\oplus u(1)\oplus sp(n+1) \subset u(1)\oplus sp(n+2)$,
\item[(B12)] $\Delta u(1)\oplus sp(1)\oplus u(1) \subset u(1)\oplus u(1)\oplus sp(2) \subset u(1)\oplus su(5)$.
\end{enumerate}
Both are ruled out by Lemma~\ref{L:nobase}.
\end{proof}


\bibliographystyle{amsplain}

\end{document}